\documentclass[reqno]{amsart}

\usepackage{amssymb,latexsym}
\usepackage{amsfonts}
\usepackage{dsfont}
\usepackage[pdftex]{hyperref}
\usepackage{color}
\usepackage{graphicx}
\usepackage[all]{xy}

\theoremstyle{plain}
\newtheorem{thm}{Theorem}[section]
\newtheorem{lem}[thm]{Lemma}
\newtheorem{prop}[thm]{Proposition}

\newtheorem*{THM}{Theorem}
\newtheorem{cor}[thm]{Corollary}
\newtheorem{rem}[thm]{Remark}
\newtheorem{defn}[thm]{Definition}

\numberwithin{equation}{section}

\newcommand{\Z}{\mathbb Z}

\newcommand{\C}{\mathbb C}
\newcommand{\fg}{\mathfrak{g}}

\newcommand{\fb}{\mathfrak{b}}
\newcommand{\fl}{\mathfrak{l}}

\newcommand{\fm}{\mathfrak{m}}
\newcommand{\ft}{\mathfrak{t}}

\newcommand{\fu}{\mathfrak{u}}
\newcommand{\fq}{\mathfrak{q}}
\newcommand{\fp}{\mathfrak{p}}

\newcommand{\B}{\mathcal{B}}

\newcommand{\V}{\mathcal{V}}

\begin{document}

\title[Affine pavings of Hessenberg varieties]{Affine pavings of Hessenberg varieties for semisimple groups}
\author[M. Precup]{Martha Precup}
\address{Department of Mathematics\\University of Notre Dame\\ Notre Dame, IN 46556}
\email{mprecup@nd.edu}

%
%
%

\begin{abstract}
In this paper we consider certain closed subvarieties of the flag variety, known as Hessenberg varieties.  We prove that Hessenberg varieties corresponding to nilpotent elements which are regular in a Levi factor are paved by affines.  We provide a partial reduction from paving Hessenberg varieties for arbitrary elements to paving those corresponding to nilpotent elements.  As a consequence, we generalize results of Tymoczko asserting that Hessenberg varieties for regular nilpotent and arbitrary elements of $\mathfrak{gl}_n(\C)$ are paved by affines.  For example, our results prove that any Hessenberg variety corresponding to a regular element is paved by affines.  As a corollary, in all these cases the Hessenberg variety has no odd dimensional cohomology.
\end{abstract}
\maketitle
\section{Introduction and Results}

\noindent This paper investigates the topological structure of Hessenberg varieties, a family of subvarieties of the flag variety introduced in \cite{dMPS}.  We prove that under certain conditions Hessenberg varieties over a complex, linear, reductive algebraic group $G$ have a paving by affines.  This paving is given explicitly by intersecting these varieties with the Schubert cells corresponding to a particular Bruhat decomposition, which form a paving of the flag variety.  This result generalizes results of J. Tymoczko in \cite{T,T2,T3}.  

Let $G$ be a linear, reductive algebraic group over $\C$, $B$ a Borel subgroup, and let $\fg$, $\fb$ denote their respective Lie algebras.  A Hessenberg space $H$ is a linear subspace of $\fg$ that contains $\fb$ and is closed under the Lie bracket with $\fb$.  Fix an element $X\in \fg$ and a Hessenberg space $H$.  The Hessenberg variety, $\B(X,H)$, is the subvariety of the flag variety $G/B=\B$ consisting of all $g\cdot \fb$ such that $g^{-1}\cdot X\in H$ where $g\cdot X$ denotes the adjoint action $Ad(g)(X)$.  

We say that a nilpotent element $N$ of a reductive Lie algebra $\fm$ is a regular nilpotent element in $\fm$ if $N$ is in the dense adjoint orbit within the nilpotent elements of $\fm$.  Suppose $N$ is a regular nilpotent element in a Levi subalgebra $\fm$ of $\fg$.  In this case, we prove that there is a torus action on $\B(N,H)$ with a fixed point set consisting of a finite collection of points.  This action yields a vector bundle over each fixed point, giving an affine paving of $\B(N,H)$ by its intersection with the Schubert cells paving $\B$.   Our argument is inspired by the proof by C. De Concini, G. Lusztig and C. Procesi that Springer fibers are paved by affines \cite{dCLP}.  The main result is as follows.

\begin{THM}  Fix a Hessenberg space $H$ with respect to $\fb$.  Let $N\in \fg$ be a nilpotent element such that $N$ is regular in some Levi subalgebra $\fm$ of $\fg$.  Then there is an affine paving of $\B(N,H)$ given by the intersection of each Schubert cell in $\B$ with $\B(N,H)$. 
\end{THM}

Theorem \ref{paving} below gives the complete statement of this result.  This generalizes Theorem 4.3 in \cite{T2} which states that the Hessenberg variety $\B(N,H)$ is paved by affines when $N\in \fg$ is a regular nilpotent element.  Moreover, we can extend this result to the Hessenberg variety $\B(X,H)$ corresponding to the arbitrary element $X\in \fg$ when $X$ is semisimple or the nilpotent part of $X$ in its Jordan decomposition satisfies the conditions of the main theorem (see Theorem \ref{pavingX}).  This implies that $\B(X,H)$ is paved by affines for all regular elements $X$.  We are therefore able to extend Tymoczko's result that the Hessenberg variety is paved by affine cells from all elements in $\mathfrak{gl}(n,\C)$, given in \cite{T}, to certain elements of an arbitrary linear, reductive Lie algebra.  Although our results are greatly influenced by results of Tymoczko, our proofs use a different approach.

The second section of this paper covers background information and facts used in the following sections.  In the third, we prove a key lemma which states that in certain cases the intersection of the Hessenberg variety $\B(X,H)$ with each Schubert cell is smooth.  Section 4 consists primarily of the statement and proof of Theorem \ref{paving}.  Last, we consider the case in which $X\in \fg$ is an arbitrary element with Jordan decomposition $X=S+N$ in section \ref{arbitrary}.  As a corollary of the results in this section, we compute the dimensions of the affine cells paving $\B(X,H)$ when $X$ is semisimple and when $X$ is an arbitrary regular element of $\fg$.

The author would like to thank Sam Evens for suggesting this problem and giving many valuable comments.  The work for this project was partially supported by the NSA.


\section{Preliminaries}

We state results and definitions from the literature which will be used in later sections.  All algebraic groups in this paper are assumed to be complex and linear.  Let $G$, $\fg$, and $\B$ be as in the section above.  


\subsection{Notation.}  In each section, we fix a standard Borel subgroup and call it $B$.  Let $T\subset B$ be a fixed maximal torus with Lie algebra $\ft$ and denote by $W$ the Weyl group associated to $T$.  Fix a representative $\dot{w}\in N_G(T)$ for each Weyl group element $w\in W= N_G(T)/T$.  Let $\Phi^+$, $\Phi^-$ and $\Delta$ denote the positive, negative, and simple roots associated to the previous data.  Let $\fg_{\gamma}$ denote the root space corresponding to $\gamma\in\Phi$ and fix a generating root vector $E_{\gamma}\in \fg_{\gamma}$.  Write $U$ for the maximal unipotent subgroup of $B$, $U^-$ for its opposite subgroup, and $\fu$ and $\fu^-$ for their respective Lie algebras.  

Given a standard parabolic subgroup $Q$ of $G$ with Levi decomposition $MU_Q$, we denote the Lie algebras of $Q$, $M$ and $U_Q$ by $\fq$, $\fm$ and $\fu_Q$ respectively.  Then $B_M:=B\cap M$ is a standard Borel subgroup of $M$ with Lie algebra $\fb_M:=\fb\cap \fm$.  Since $Q$ is standard, $M$ corresponds to a subset $\Delta_M$ of simple roots.  Denote by $\Phi(\fu_Q)$ and $\Phi_M$ the subsets of roots so that
		\[
				\fu_Q = \bigoplus_{\gamma\in \Phi(\fu_Q)} \fg_{\gamma} \quad \mbox{and} \quad 																\fm=\ft\oplus \bigoplus_{\gamma\in\Phi_M} \fg_{\gamma}.
		\]
In particular, $\fm$ has triangular decomposition $\fm=\fu_M^- \oplus \ft \oplus \fu_M$ where
		\[
				\fu_M=\bigoplus_{\gamma\in \Phi^+_M} \fg_{\gamma} \quad \mbox{and} \quad 														\fu_M^- = \bigoplus_{\gamma\in \Phi_M^-} \fg_{\gamma},
		\]
with $\Phi_M^{\pm}=\Phi_M \cap \Phi^{\pm}$.  Let $U_M$ denote the unipotent subgroup of $G$ with Lie algebra $\fu_M$.  Then $U_M$ is the maximal unipotent subgroup of $B_M$, and $\fu=\fu_M\oplus \fu_Q$.


\subsection{Hessenberg Varieties.}  We give the precise definition of a Hessenberg variety.

\begin{defn}  A subspace $H\subseteq \fg$ is a {\it Hessenberg space} with respect to $\fb$ if $\fb\subset H$ and $H$ is a $\fb$-submodule.
\end{defn}

Denote by $\Phi_H\subseteq \Phi$ the subset of roots such that $H=\ft \oplus \bigoplus_{\gamma \in\Phi_H} \fg_{\gamma}$.  Then the conditions that $H$ is a Hessenberg space are equivalent to requiring that $\Phi^+\subseteq \Phi_H$ and $\Phi_H$ be closed with respect to addition with roots from $\Phi^+$.  Let $X\in \fg$ and $H$ be some fixed Hessenberg space.  Set
		\[
				G(X,H)=\{ g\in G : g^{-1}\cdot X\in H \}
		\]
where $g\cdot X$ denotes $Ad(g)(X)$.  Then $G(X,H)$ is a subvariety of $G$ which is invariant under right multiplication by $B$.  Let
		\[
				\B(X,H)=\{ g\cdot \fb \in \B : g\in G(X,H)    \}
		\]
denote its image in the flag variety $\B$.  This is the Hessenberg variety associated to $X$ and $H$.  Note that when $H=\fb$, $\B(X,H)$ is the variety of Borel subalgebras containing $X$, denoted $\B^X$, and called the Springer variety of $X$.  In the other extreme, when $H=\fg$, the Hessenberg variety is the full flag variety $\B$.

\begin{defn}  We say $X\in \fg$ is in standard position with respect to $(\fb,\ft)$ if $X=S+N$ with $S\in \ft$ and $N\in \fu$.
\end{defn}

\begin{rem}  For any $X\in \fg$, there exists $g\in G$ so that $g\cdot X$ in standard position with respect to $(\fb,\ft)$.  Since the map $l_g: \B \to \B$, given by $l_g(a\cdot \fb) = ga\cdot \fb$ induces an isomorphism $l_g : \B(X,H) \to \B(g\cdot X,H)$, we may always assume $X$ is in standard position with respect to $(\fb,\ft)$.
\end{rem}


\subsection{Pavings.}
In what follows we show that for certain elements $X\in \fg$, $\B(X,H)$ is paved by affines.

\begin{defn}  A paving of an algebraic variety $Y$ is a filtration by closed subvarieties
		\[
				Y_0\subset Y_1 \subset \cdots \subset Y_i \subset \cdots \subset Y_d=Y.
		\]	
A paving is affine if $Y_i-Y_{i-1}$ is a finite, disjoint union of affine spaces.
\end{defn}

There is a well known affine paving of the flag variety given by the Bruhat decomposition.  Indeed, $\B=\bigsqcup_{w\in W}X_w$ where each $X_w=B\dot{w}B/B$ denotes the Schubert cell indexed by $w\in W$.  Each $X_w$ has the following explicit description.

\begin{lem} \label{S cells} Fix $w\in W$.  The following are isomorphic varieties:
		\begin{enumerate}
				\item the Schubert cell $X_w= B\dot{w}B/B$;
				\item the subgroup $U^w=\{ u\in U : \dot{w}^{-1} u \dot{w} \in U^- \}$; and 
				\item the Lie subalgebra, $\fu^w:= Lie(U^w) =\bigoplus_{\alpha\in\Phi_w}\fg_{\alpha}$ where 											$\Phi_w=\{ \gamma \in\Phi^+ : w^{-1}(\gamma)\in \Phi^- \}$.  In particular, $\dim U^w=|\Phi_w|$.
		\end{enumerate}
\end{lem}

Additionally, $\overline{X_w}=\bigsqcup_{w'\leq w} X_{w'}$ where $\leq$ denotes the Bruhat order on the Weyl group (see \cite{BGG}).  Set $\B_i=\bigsqcup_{w\in W;\; |\Phi_w|=i} \overline{X_w}$.  Then the $\B_i$ are closed subvarieties of $\B$ which give an affine paving of $\B$ since
		\[
				\B_i-\B_{i-1} = \bigsqcup_{w\in W;\; |\Phi_w|=i} X_w \cong \bigsqcup_{w\in W; \;  |\Phi_w|=i} \fu^w 														\cong \bigsqcup_{w\in W; \;  |\Phi_w|=i}\C^i.
		\]
The Hessenberg variety $\B(X,H)$ is a closed subvariety of $\B$, so the intersections $\B_i\cap \B(X,H) = \bigsqcup_{w\in W;\;  |\Phi_w|=i} \overline{X_w}\cap \B(X,H)$ are closed.  They form a paving of $\B(X,H)$ where
		\[
				\B_i \cap \B(X,H) - \B_{i-1}\cap \B(X,H) = \bigsqcup_{w\in W;\; |\Phi_w|=i} X_w\cap \B(X,H).
		\]
To show this paving is affine, we will show that each $X_w\cap \B(X,H)$ is homeomorphic to some affine space $\C^{d}$.  In summary,
\begin{rem} \label{paveit}  $\B(X,H)$ is paved by the intersections $\B_i\cap \B(X,H)$ and therefore paved by affines if $X_w\cap \B(X,H)\cong \C^d$ for all $w\in W$ and some $d\in \Z$.
\end{rem}  

Using the identification $X_w\cong U^w$, we can write the intersection explicitly as
		\[
				X_w\cap \B(X,H)=\{ u\dot{w}\cdot \fb : u\in U^w,\; u^{-1}\cdot X \in \dot{w} \cdot H \}.
		\]
A paving by affine cells computes the Betti numbers of an algebraic variety $Y$.

\begin{lem}  Let $Y$ be an algebraic variety with an affine paving, $Y_0\subset Y_1 \subset \cdots \subset Y_i \subset \cdots \subset Y_d=Y$.  Then the nonzero compactly supported cohomology groups of $Y$ are given by $H_c^{2k}(Y)= \Z^{n_k}$ where $n_k$ denotes the number of affine components of dimension $k$.
\end{lem}


\subsection{Associated Parabolic}\label{parab}   Let $N\in \fg$ be a nonzero nilpotent element.  By the Jacobson-Morozov theorem (\cite{CG}, Theorem 3.7.4) there exists a homomorphism of algebraic groups $\phi: SL_2(\C) \to G$ such that
		\[
				d\phi\begin{pmatrix} 0 & 1\\ 0 & 0 \end{pmatrix}=N.
		\]
Define a 1-parameter subgroup $\lambda: \C^* \to G$ so that $\lambda(z)= \phi \begin{pmatrix} z & 0 \\ 0 & z^{-1} \end{pmatrix}$ for all $z\in \C^*$, and consider the $\lambda$-weight spaces of $\fg$:
		\[
				\fg_i(\lambda)=\{ X\in \fg : \lambda(z)\cdot X=z^iX \;\; \forall z\in \C^* \}.
		\]
When there is no ambiguity we write $\fg_i$ instead of $\fg_i(\lambda)$.  Now, $N\in \fg_2$ and we can decompose $\fg$ as $\fg=\bigoplus_{i\in \Z} \fg_i$ where $[\fg_i,\fg_j]\subset \fg_{i+j}$ for all $i,j\in \Z$.  Let $L$ and respectively $P$ denote the connected algebraic subgroups of $G$ whose Lie algebras are $\fl:=\fg_0$ and $ \fp:= \bigoplus_{i\geq 0}\fg_i$.  It is known that
		\begin{enumerate}
				\item $P$ is a parabolic subgroup depending only on $N$ (not on the choice of $\phi$).
				\item $P=L U_P$ is a Levi decomposition and its unipotent radical $U_P$ has Lie algebra 									$\fu_P=\bigoplus_{i>0}\fg_i$.
		\end{enumerate}
				

\begin{lem}\label{onto}  The maps 
		\[
				ad_N: \fu_P \to \bigoplus_{i\geq 3} \fg_i \quad \mbox{and} \quad ad_N: \fl \to \fg_2
		\]
are onto.
\end{lem}

Generally, a 1-parameter subgroup $\lambda: \C^* \to T$ is {\it dominant} with respect to $\Phi^+$ if $\left< \gamma, \lambda \right> \geq 0$ for all $\gamma\in \Phi^+$.  Here $\left<\,,\,\right>$ is the natural pairing between the character and cocharacter groups of $G$ defined by $\lambda(z) \cdot E_{\gamma} = z^{\left< \gamma,\lambda \right>} E_{\gamma}$.  If $\lambda$ is the 1-parameter subgroup associated to nilpotent element $N$ as above, then $\lambda$ is dominant if and only if $P$ is a standard parabolic subgroup.


\subsection{A key lemma.}  There is a result yielding a vector bundle structure which we will use in the following sections.  It is a special case of Theorem 9.1 in \cite{BH}.

\begin{lem}\label{vb1}  Let $\pi: E \to Y$ be a vector bundle over a smooth variety $Y$ with a fiber preserving a linear $\C^*$-action on $E$ with strictly positive weights.  Let $E_0\subset E$ be a $\C^*$-stable, smooth, closed subvariety.  Then the restriction $\pi: E_0\to \pi(E_0)$ is a vector sub-bundle of $\pi: E \to Y$.
\end{lem}


\section{Fixed Point Reduction} \label{fixedpts}


Let $Q$ be a standard parabolic subgroup of $G$ with Levi decomposition $Q=MU_Q$.  The Levi subgroup $M$ is a connected, reductive algebraic group.  Thus its connected centralizer $Z:=Z_G(M)^0\subset T$ is a torus.  Consider the action of $Z$ on the flag variety, $\B$.  We can explicitly calculate the fixed point set $\B^Z$ using the following.

\begin{prop} \label{cg} (\cite{CG}, Proposition 8.8.7)  Each connected component of $\B^Z$ is isomorphic to the flag variety of $M$, $\B(M)$.  In particular, the connected component containing $\fb_0\in \B^Z$ is $M\cdot \fb_0 \cong M/(M\cap B_0)$ where $B_0$ is the Borel subgroup of $G$ such that $Lie(B_0)=\fb_0$ and $M\cap B_0$ is a Borel subgroup of $M$.
\end{prop}

There exists a 1-parameter subgroup $\mu: \C^* \to Z$ so that the $\mu$-fixed points and $Z$-fixed points of $\B$ coincide (\cite{H}, 25.1).  Every 1-parameter subgroup in $T$ is $W$-conjugate to a dominant 1-parameter subgroup so without loss of generality we may assume $\fm=\fg_0(\mu)$ and $\fu_Q=\bigoplus_{i>0} \fg_i(\mu)$.  

Recall that given a standard Levi subgroup $M$ of $G$, we write $W_M$ to denote the subgroup of the Weyl group associated to $M$.  Let
		\[
				W^M = \{ v\in W : \Phi_v\subseteq \Phi(\fu_Q) \},
		\]
where $\Phi_w=\{ \gamma\in \Phi^+ : w^{-1}(\gamma)\in \Phi^- \}$.  The elements of $W^M$ form a set of minimal representatives for $W_M \backslash W$ in the following sense.

\begin{lem} (\cite{K}, Proposition 5.13) \label{decomp}  Each $w\in W$ can be written uniquely as $w=yv$ with $y\in W_M$ and $v\in W^M$ such that $l(w)=l(y)+l(v)$.
\end{lem}

\begin{cor} (\cite{K}, equation (5.13.2))  \label{roots}  Let $w=yv$ be the decomposition of $w\in W$ given above.  Then $\Phi_w=y(\Phi_v) \bigsqcup \Phi_y$.
\end{cor}

Consider the Schubert cell $X_w\cong U^w$.  Suppose $w\in W$ has decomposition $w=yv$ with $y\in W_M$ and $v\in W^M$.  Then by Corollary \ref{roots}
		\[
				U^w \cong \fu^w \cong \bigoplus_{\gamma\in y(\Phi_v)} \fg_{\gamma} 														\oplus \bigoplus_{\gamma\in \Phi_y} \fg_{\gamma} = \dot{y}\cdot \fu^v \oplus \fu^y.
		\]
Now, $\mu$ yields a $\C^*$-action on $\fu_{Q}$ with strictly positive weights.  Therefore, $(U^w)^{\mu} \cong (\dot{y}\cdot \fu^v)^{\mu}\oplus (\fu^y)^{\mu}=\fu^y \cong U^y$, since $\fu^y \subset \fm=\fg_0(\mu)$ and $\dot{y}\cdot \fu^v \subset \fu_Q$.  

\begin{rem} \label{isom}  The isomorphism of each connected component of $\B^Z$ with $\B(M)$ given in Proposition \ref{cg} can be described explicitly on each Schubert cell $X_w$ by
		\begin{eqnarray*}
				X_w^{Z}=X_w^{\mu} \to X_y; \;  u\dot{y}\dot{v}\cdot \fb \mapsto u \dot{y}\cdot \fb_M
		\end{eqnarray*}
for all $u\in U^y$.
\end{rem}  

Since $X_w\cong \dot{y}\cdot \fu^v \oplus \fu^y \cong \dot{y}\cdot \fu^v \times X_w^{\mu}$ we get a trivial vector bundle structure
		\begin{eqnarray} \label{vb}
				\xymatrix{\dot{y}\cdot \fu^v \times X_w^{\mu} \ar[r] & X_w \ar[d]^{\pi_{\mu}}\\  & X_w^{\mu} }
		\end{eqnarray}
where the base space $X_w^{\mu}$ can be naturally identified with the Schubert cell in the flag variety of $M$ corresponding to $y\in W_M$.

\begin{rem}  The fiber of the vector bundle $\pi_{\mu}: X_w \to X_w^{\mu}$ is a subset of $\fu_Q$, so the $\C^*$-action induced by $\mu$ acts with strictly positive weights on the fiber.
\end{rem}

\begin{rem} \label{cases} If $Q$ is a Borel subgroup, then $Z$ is a maximal torus and the corresponding 1-parameter subgroup $\mu: \C^* \to T$ is regular with respect to $\Phi$, i.e. $\left< \alpha,\mu \right>\neq 0$ for all $\alpha\in \Phi$.  In this case, $X_w^{\mu}=\{ \dot{w}\cdot \fb \}$ and the fiber of $\pi_{\mu}$ is $\fu^w$.
\end{rem}

We will show that for certain elements $X\in \fg$, the intersection $X_w \cap \B(X,H)$ is affine for all $w\in W$.  Our general method of proof will be to apply Lemma \ref{vb1} to the vector bundle in equation ($\ref{vb}$).  To apply the Lemma, however, we need to show that the intersection $X_w\cap \B(X,H)$ is smooth.  We can do this provided we have some understanding of the Adjoint $U$-orbit of $X$ in $\fg$, $U\cdot X$.

\begin{prop} (see \cite{dCLP}, Proposition 3.2) \label{smooth} Let $X\in \fg$ have Jordan decomposition $X=S+N$, and assume $X$ is in standard position with respect to $(\fb,\ft)$.  Suppose $U\cdot X = X + \V$ where $\V = \bigoplus_{\gamma\in \Phi(\V)} \fg_{\gamma} \subset \fu$ and $X\notin \V$.  Fix a Hessenberg space $H$ of $\fg$ with respect to $\fb$.  Then for all $w\in W$, $X_w\cap \B(X,H)\neq \emptyset$ if and only if $N\in \dot{w}\cdot H$.  If $X_w\cap \B(X,H)\neq \emptyset$ then it is smooth and $\dim \left( X_w\cap \B(X,H) \right) =|\Phi_w| - \dim \V / (\V \cap \dot{w}\cdot H)$.   
\end{prop}
\begin{proof}  First, we identify the nonempty intersections.  Note that $S\in \dot{w}\cdot H$ for all $w\in W$ since $S \in \ft\subset \dot{w}\cdot H$.  Thus if $N\in \dot{w}\cdot H$, then $X=N+S\in\dot{w}\cdot H$ and $X_w\cap \B(X,H)$ is nonempty.  Alternatively, say $u\dot{w}\cdot \fb\in X_w\cap \B(X,H)$ for some $u\in U$ where $u^{-1}\cdot X=X+Y$ with $Y\in \V$.  Write $N=\sum_{\gamma\in \Phi^+} c_{\gamma}E_{\gamma}$ and $Y=\sum_{\gamma\in \Phi^+} d_{\gamma}E_{\gamma}$ for some $c_{\gamma},d_{\gamma}\in \C$.  Since $X\notin \V$ and $Y\in \V$, $c_{\gamma}=0$ for all $\gamma\in \Phi(\V)$ and $d_{\gamma}=0$ for all $\gamma\notin \Phi(\V)$.  Therefore
		\[
				u^{-1}\cdot X\in \dot{w}\cdot H \Rightarrow X+Y\in \dot{w}\cdot H \Rightarrow S+N+Y\in \dot{w}\cdot H 															\Rightarrow N\in \dot{w}\cdot H
		\]
since $N$ and $Y$ do not have any components in common with respect to this root space decomposition.

Suppose $X_w\cap \B(X,H)\neq \emptyset$.  The stabilizer of $\dot{w} \cdot \fb$ in $U$ is $U_w=\dot{w} U \dot{w}^{-1} \cap U$ and $U^w \cong U/U_w$.  Now
		\[
				X_w \cap \B(X,H)=\{ u\dot{w}\cdot \fb : u^{-1}\cdot X \in \dot{w}\cdot H\}\subset U/U_w.
		\]
Since $X_w \cap \B(X,H)$ is the image of 
		\[
				U(X,\dot{w}\cdot H)=\{ u\in U : u^{-1}\cdot X \in \dot{w}\cdot H \}		
		\]
under the quotient map $U\to U/U_w$, it is enough to show that $U(X,\dot{w}\cdot H)$ is smooth and has dimension $\dim U - \dim \V /(\V \cap \dot{w}\cdot H)$.  

Consider the morphism $\phi: U \to X+ \V$ given by $\phi(u)=u^{-1}\cdot X$.  Since $U \cdot X=X+\V$, $\phi$ can be identified with the quotient morphism $U\to U/Z_U(X)$, where $Z_U(X)$ denotes the centralizer of $X$ in $U$, and is therefore a smooth morphism of relative dimension $\dim Z_{U}(X)$.  Let $i: \V\cap \dot{w}\cdot H \to X+\V$ be the map given by $Y\mapsto X+Y$ for all $Y\in \V \cap \dot{w}\cdot H$.  By \cite{Ha}, Proposition III.10.1(b), the morphism $\tilde{\phi}$ induced by the base change given in the Cartesian diagram
		\[
				\xymatrix{ U \times_{\V} (\V\cap \dot{w}\cdot H) \ar[d]_{\tilde{\phi}} \ar[r] & U \ar[d]^{\phi}\\
							\V\cap\dot{w}\cdot H  \ar[r]^{i} & X+\V
						}
		\]
is smooth of relative dimension $\dim  Z_U(X)$.  Since $\V\cap \dot{w}\cdot H \subset \fg$ is a linear subspace it is a smooth variety, and the projection of $\V\cap \dot{w}\cdot H$ onto a point is smooth of relative dimension $\dim \left( \V\cap \dot{w}\cdot H\right)$.  Since the composition of smooth morphisms is smooth (\cite{Ha}, Proposition III.10.1(c)), $U\times_\V (\V\cap \dot{w}\cdot H)$ is smooth and has dimension $\dim Z_U(X) + \dim (\V \cap \dot{w}\cdot H)$.  But
		\begin{eqnarray*}
				U\times_{\V} (\V\cap \dot{w}\cdot H) &=& \{ (u,Y) \in U\times (\V\cap \dot{w}\cdot H) : \phi(u)=i(Y) \}\\
							&\cong& \{ u\in U : u^{-1}\cdot X=X+Y \in X+ (\V \cap \dot{w}\cdot H) \}\\
							&=& \{ u\in U : u^{-1}\cdot X\in \dot{w}\cdot H \}\\
							&=& U(X,\dot{w}\cdot H).
		\end{eqnarray*}
Thus $U(X,\dot{w}\cdot H)$ is indeed smooth and has dimension
		\begin{eqnarray*}
				\dim Z_U(X)+\dim (\V\cap \dot{w}\cdot H) &=& \dim U- \dim \V + \dim (\V\cap \dot{w}\cdot H) \\															         &=& \dim U - \dim \V/(\V\cap \dot{w}\cdot H)
		\end{eqnarray*}
as required. 
\end{proof}


\section{The nilpotent case, $N\in \fg$} \label{Nilpotent}

In the section above, we proved that when the $U$-orbit through $X\in \fg$ is an affine subspace of $\fg$ then $X_w \cap \B(X,H)$ is smooth.  This will allow the application of Lemma \ref{vb1} to the vector bundle in equation ($\ref{vb}$).  In this section we do this for a nilpotent element $N\in \fg$ which is regular in some Levi subalgebra $\fm$ of $\fg$.  To understand the orbit $U\cdot N$ we utilize of the theory of the parabolic subgroup associated to $N$.

Suppose $N$ is nilpotent and regular in a Levi subalgebra $\fm$ of $\fg$ corresponding to Levi subgroup $M$ of $G$.  Since $N$ is regular in $\fm$ it is conjugate to a sum of simple root vectors in $\fm$.  Fix a standard Borel subalgebra $\fb$ with respect to which $\fm$ is standard and 
		\[
				N = \sum_{\alpha\in \Delta_M} E_{\alpha}.
		\]  
Let $\tilde{\lambda}:\C^* \to T$ be the 1-parameter subgroup associated to $N$ as in section \ref{parab}.  Note that $\tilde{\lambda}$ may not be dominant with respect to $\Phi^+$, but there exists $w_1\in W$ such that $\dot{w}_1\cdot \tilde{\lambda}$ is dominant.  Let $P$ be the standard parabolic subgroup whose Lie algebra is $\fp=\fl \oplus \fu_P$ where $\fl=\fg_0(\dot{w}_1 \cdot \tilde{\lambda})$ and $\fu_P=\oplus_{i> 0}\fg_i(\dot{w}_1\cdot \tilde{\lambda}) $.  Let $L$ be the Levi subgroup of $G$ with Lie algebra $\fl$.

\begin{lem} \label{dominant}  If $\dot{w}_1\cdot \tilde{\lambda}$ is dominant then $\dot{v}_1\cdot \tilde{\lambda}$ is dominant, where $w_1=y_1v_1$ with $y_1\in W_L$ and $v_1\in W^L$.
\end{lem}
\begin{proof}  By assumption $\left< \gamma, \dot{w}_1\cdot \tilde{\lambda} \right> \geq 0$ for all $\gamma\in \Phi^+$.  Recall that if $\gamma\in \Phi(\fu_P)$ then $y_1(\gamma)\in \Phi(\fu_P)$ and if $\gamma\in \Phi^+_L$ then $y_1(\gamma)\in \Phi_L$.  Thus for all $\gamma\in \Phi^+$
		\[
				\left< \gamma, \dot{v}_1\cdot \tilde{\lambda} \right> 																		= \left< y_1(\gamma), \dot{y}_1\dot{v}_1\cdot \tilde{\lambda} \right> 													= \left< y_1(\gamma) , \dot{w}_1\cdot \tilde{\lambda} \right> \geq 0,
		\]
so $\dot{v}_1\cdot \tilde{\lambda}$ is dominant.
\end{proof}

Set $\lambda:= \dot{v}_1\cdot \tilde{\lambda}$.  Note that $\lambda$ defines the same parabolic subgroup $P$ as $y_1\cdot \lambda=w_1\cdot \tilde{\lambda}$, since conjugation by an element of $W_L$ preserves $P$.  Therefore we can replace $P$ by its $y_1^{-1}$-conjugate, i.e. we let $\fp=\fl\oplus \fu_P$ where $\fl=\fg_0(\lambda)$ and $\fu_P=\bigoplus_{i>0}\fg_i(\lambda)$.  Replace $N$ by its $v_1$-conjugate.  Abusing notation we denote it by $N$, so
		\[
				N=\sum_{\alpha\in \Phi_N} E_{\alpha}
		\]
where $\Phi_N=v_1(\Delta_M)$.  Then $\lambda$ is the 1-parameter subgroup associated to $N$ and $P$ is the standard parabolic subgroup associated to $N$.

\begin{rem}\label{regpsg}  Since $N\in \fg_2(\lambda)$, $\left< \alpha,\lambda \right>=2$ for all $\alpha\in \Phi_N$.  Therefore $\left< \gamma, \lambda \right> \geq 2$ for all $\gamma\in v_1(\Phi_M)$; in particular, $\Phi_L \cap v_1(\Phi_M)= \emptyset$.
\end{rem}

Define
		\[
				\Phi^+(V)=\{\gamma\in\Phi^+ : \gamma=\alpha+\beta \mbox{  for some  } \alpha \in \Phi_N\mbox{  and  } 													\beta\in \Phi_L^+ \}
		\] 
and let $V=\bigoplus_{\gamma\in \Phi^+(V)} \fg_{\gamma}$.  Then $V\subset \fg_{2}(\lambda)$ is well-defined subspace with respect to this basis.  Similarly, let
		\[
				\Phi^-(V)=\{ \gamma\in\Phi^+ : \gamma=\alpha+\beta \mbox{  for some  } \alpha \in\Phi_N \mbox{  and  } 													\beta\in\Phi_L^- \}
		\]
and $V^-=\bigoplus_{\gamma\in \Phi^-(V)} \fg_{\gamma}$.  Our reason for defining these subspaces of $\fg_2(\lambda)$ is to analyze the adjoint action of $N$ on $\fl=\fu_L^-\oplus \ft \oplus \fu_L$.  Indeed, given $E_{\beta}\in \fg_{\beta}\subset \fu_L$ we have
		\[
				ad_N\, E_{\beta}=[N,E_{\beta}]=\sum_{\alpha\in \Phi_N} [E_{\alpha},E_{\beta}]\in V
		\]
since $[E_{\alpha},E_{\beta}]\in \fg_{\alpha+\beta}$ whenever $\alpha+\beta\in \Phi$.  Similarly for all $E_{\beta}\in \fg_{\beta}\subset \fu_L^-$, $ad_N\,E_{\beta}\in V^-$.

\begin{rem}  \label{adaction}  We have just shown $[Y,N]\in V$ for all $Y\in \fu_L$.  By construction, $ad_Y:V \to V$ as well.
\end{rem}

\begin{lem} \label{empty}  There is a direct sum decomposition of $\fg_2(\lambda)$, 
		\[
				\fg_2(\lambda)= V^- \oplus \bigoplus_{\alpha\in \Phi_N} \fg_{\alpha} \oplus V.
		\]
\end{lem}
\begin{proof}  By Lemma \ref{onto}, the map $ad_N: \fl \to \fg_2(\lambda)$ is onto.  Since $ad_N(\fu_L) \subset V$, $ad_N(\fu_L^-)\subset V^-$ and $ad_N(\ft)\subset \bigoplus_{\alpha\in \Phi_N} \fg_{\alpha}$, it is certainly the case the $\fg_2(\lambda)$ is a sum of these subspaces.  We must show that their pairwise intersection is $\{0\}$.  To do so, we show that the corresponding subsets of roots are pairwise disjoint.

First, suppose there exists $\gamma\in \Phi^+(V)\cap \Phi^-(V)$.  Then there are roots $\alpha_1,\alpha_2\in \Phi_N$ and $\beta_1,\beta_2\in \Phi^+_L$ so that
		\[
				\alpha_1+\beta_1=\gamma=\alpha_2-\beta_2.
		\]
Recall that $\Phi_N=v_1(\Delta_M)$, so we rewrite this equality as
		\[
				v_1^{-1}(\alpha_1)+v_1^{-1}(\beta_1)= v_1^{-1}(\alpha_2)-v_1^{-1}(\beta_2).
		\]
Since $v_1\in W^L$ and $\beta_1,\beta_2 \in \Phi^+_L$, we get that $v_1^{-1}(\beta_1),v_1^{-1}(\beta_2) \in \Phi^+$.  By assumption $v_1^{-1}(\alpha_1)$ and $v_1^{-1}(\alpha_2)$ are simple roots.  Therefore $v_1^{-1}(\alpha_1)+v_1^{-1}(\beta_1)\in \Phi^+$ and $v_1^{-1}(\alpha_2)-v_1^{-1}(\beta_2)\in \Phi^-$.  The two cannot be equal, giving a contradiction.

Similarly, suppose $\gamma\in \Phi^+(V)\cap \Phi_N$.  Then there exists $\alpha_1\in \Phi_N$ and $\beta_1\in \Phi_L^+$ so that $\gamma=\alpha_1+\beta_1$, implying $v_1^{-1}(\gamma)=v_{1}^{-1}(\alpha_1)+v_1^{-1}(\beta_1)$.  Since $v_1\in W^L$, $v_1^{-1}(\beta_1)\in \Phi^+$ and by assumption $v_1^{-1}(\gamma)$ and $v_1^{-1}(\alpha_1)$ are simple.  But this means that simple root $v_1^{-1}(\gamma)$ can be written as the sum of positive roots $v_1^{-1}(\alpha_1)$ and $v_1^{-1}(\beta_1)$, which is a contradiction.

Finally, $\Phi^-(V)\cap \Phi_N=\emptyset$ by a similar argument.
\end{proof}

Recall that our goal is to understand the Adjoint $U$-orbit of $N$, $U\cdot N$.  To do so, we need a few facts about unipotent groups.  Let $\tilde{U}$ be a unipotent subgroup with Lie algebra $\tilde{\fu}$.  First, since $\tilde{U}$ is unipotent the exponential map $exp: \tilde{\fu} \to \tilde{U}$ is a diffeomorphism.  Recall that for all $Y\in \tilde{\fu}$,
		\[
				exp(Y) \cdot X= X+[Y,X]+\frac{1}{2} [Y,[Y,X]] + \cdots = X+ \sum_{i=1}^{\infty} \frac{ad_Y^i(X)}{i!}.
		\]
Thus if $ad_Y: \V \to \V\subset \fg$ and $[Y,X]\in \V$ for all $Y\in \tilde{\fu}$, we get $\tilde{U}\cdot X \subset X+\V$.

Next, suppose $\tilde{U}$ acts on an irreducible affine variety $Y$.  Given $y\in Y$ the $\tilde{U}$-orbit of $y$ is closed (\cite{H}, Exercise 17.8).  Therefore if the dimension of the orbit is equal to the dimension of $Y$, $Y=\tilde{U}\cdot y$.  We can apply this to our situation as follows.

\begin{rem}\label{unipotent}  Let $\tilde{U}$ be a unipotent subgroup such that $\tilde{U}\cdot X \subset X+\V \subset \fg$ for $X\in \fg$ and $\dim \tilde{U} - \dim Z_{\tilde{U}}(X)=\dim \tilde{U}\cdot X= \dim \V$.  Then $\tilde{U}\cdot X=X+\V$.
\end{rem}

\begin{lem}  Recall that $U_L$ is the unipotent subgroup of $G$ with Lie algebra $\fu_L$.  Then $U_L \cdot N = N+V$.   
\end{lem}
\begin{proof}  First, Remark \ref{adaction} implies $U_L\cdot N\subset N+V$.  By Remark \ref{unipotent}, we have only to show that $\dim U_L -\dim Z_{U_L}(N)=\dim V$.

Since $ad_N: \fl \to \fg_2$ is surjective, for all $X\in V\subset \fg_2$ there exists $Y\in \fl$ such that $[N,Y]=X$.  Using the decomposition $\fl= \fu_L^- \oplus \ft \oplus \fu_L$, there exists $Y_-\in \fu_L^-$, $S\in \ft$ and $Y_+ \in \fu_L$ such that $Y=Y_-+S+Y_+$.  Therefore
		\[
				[N,Y_-]+[N,S]+[N,Y_+]=[N,Y]=X\in V.
		\]
But $[N,Y_-]\in V^-$ and $[N,S]\in \bigoplus_{\alpha\in \Phi_N} \fg_{\alpha}$.  So by Lemma \ref{empty}, $[N,Y_-]=[N,S]=0$.  Thus for all $X\in V$ there exists an element $Y_{+}\in \fu_L$ such that
		\[
				ad_N(Y_+)=[N,Y_+]=[N,Y]=X.
		\]
Since $ad_N: \fu_L \to V$ is surjective, $\dim V= \dim \fu_L - \dim \fu_L^{N}=\dim U_L - \dim Z_{U_L}(N)$.
\end{proof}

\begin{cor} \label{Uorb} $U\cdot N = N + \V$ where $\V= V\oplus \bigoplus_{i\geq 3} \fg_i(\lambda)$.
\end{cor}
\begin{proof}  First, since $\fu=\fu_L\oplus \fu_P$, $ad_N: \fu_L \to V$, and $ad_N: \fu_P \to \bigoplus_{i\geq 3} \fg_i(\lambda)$ we conclude that $[Y, N]\in \V$ for all $Y\in \fu$.  Additionally, $ad_Y: \V \to \V$ for all $Y\in \fu$, therefore $U \cdot N \subset N+ \V$.  Note that $\fu^N=\fu_L^N\oplus \fu_P^N$ (since $V\cap \bigoplus_{i\geq 3} \fg_i(\lambda)=\{ 0 \}$) so
		\begin{eqnarray*}
				\dim U - \dim Z_U(N) &=& \dim \fu - \dim \fu^N\\
								&=& \dim \fu_L - \dim \fu_L^{N} + \dim \fu_P - \dim \fu_P^{N} \\
								&=& \dim V + \dim\; \bigoplus_{i\geq 3} \fg_i(\lambda)\\
								&=& \dim \V
		\end{eqnarray*}
where the third equality follows by Lemma \ref{onto}.  Hence $U\cdot N = N + \V$ by Remark \ref{unipotent}.
\end{proof}

\begin{rem}  \label{Nreg2}  If $N\in \fg$ is regular then $N=\sum_{\alpha\in \Delta} E_{\alpha}$ with respect to the choice of Borel subalgebra above.  Then $\tilde{\lambda}$ is dominant and regular, i.e. $\tilde{\lambda}=\lambda$.  In particular, $\fl=\fg_0(\lambda)=\ft$ so $V=\{0\}=V^-$ and $\V = \bigoplus_{i\geq 3} \fg_i(\lambda) = \bigoplus_{\gamma\in \Phi^+-\Delta} \fg_{\gamma}$.
\end{rem}

For future use, we restate Corollary \ref{Uorb} as follows.

\begin{cor} \label{Uorb3}  Suppose $N\in \fg$ is regular in some Levi subalgebra of $\fg$.  Then there exists a Borel subalgebra $\fb$ of $\fg$ so that $N\in \fu$ and $U\cdot N=N+\V$ where $\V\subset \fu$ is a direct sum of root spaces such that $N\notin \V$.
\end{cor}
\begin{proof}  Pick a Borel subalgebra $\fb\subset \fg$ so that the parabolic subgroup $P=LU_P$ associated to $N$ is standard and $N=\sum_{\alpha\in \Phi_N} E_{\alpha}$ where $\Phi_N\subset \Phi^+$ has the property that there exists $v_1\in W^L$ such that $v_1^{-1}(\Phi_N)\subset \Delta$.  Such a Borel subalgebra can be found by following the process given above.  The statement now follows from Corollary \ref{Uorb}.
\end{proof}

\begin{thm} \label{paving}  Suppose $N\in \fg$ is regular in some Levi subalgebra $\fm$ of $\fg$.  Then $\B(N,H)$ is paved by affines.\end{thm}
\begin{proof}  Let $H$ be a Hessenberg space with respect to $\fb$.  By Corollary \ref{Uorb}, the $U$-orbit of $N$ is $N+\V$ where $\V$ is a direct sum of root spaces such that $N\notin \V$.  Therefore by Proposition \ref{smooth}, $X_w\cap \B(N,H)\neq \emptyset$ if and only if $N\in \dot{w}\cdot H$ and when $X_w\cap \B(N,H)\neq \emptyset$, the intersection is smooth.  Recall that equation (\ref{vb}) exhibits a vector bundle $\pi_{\lambda}: X_w \to X_w^{\lambda}$ with a fiber preserving the strictly positive $\C^*$-action induced by $\lambda$.  In addition, $X_w\cap \B(N,H)$ is stable under this action.  Indeed, if $u\dot{w}\cdot \fb \in X_w\cap \B(N,H)$ then for all $z\in \C^*$
		\[
				(\lambda(z)u)^{-1}\cdot N= u^{-1}\cdot (\lambda(z^{-1})\cdot N)= z^{-2} u^{-1}\cdot N\in \dot{w}\cdot H.
		\]  
Applying Lemma \ref{vb1}, we get a vector sub-bundle $ \pi_{\lambda}: X_w\cap \B(N,H) \to X_w^{\lambda} \cap \B(N,H)$.  Write $w=yv$ where $y\in W_L$ and $v\in W^L$.  Then $X_w^{\lambda}\cong X_y$ where $X_y$ is the Schubert cell in $\B(L)$ corresponding to $y\in W_L$ (see Remark \ref{isom}).  

Consider the torus $Z=Z_G(v_1Mv_1^{-1})^0$.  Let $\mu: \C^* \to T$ be a dominant 1-parameter subgroup such that $\B^Z=\B^{\mu}$.  Then $\mu$ is regular with respect to $\Phi_L^+$ by Remark \ref{regpsg} and $\mu(z)\cdot N=N$ for all $z\in \C^*$ since $N\in \dot{v}_1\cdot \fm$.  Apply Remark \ref{cases} to get a vector bundle $\pi_{\mu}: X_y\to X_y^{\mu}=\{ \dot{y}\cdot \fb_L \}$ with fiber preserving a strictly positive $\C^*$-action induced by $\mu$.  For all $u\dot{y}\cdot \fb_L\in X_y\cap \B(N,H)$,
		\[
				(\mu(z)u)^{-1}\cdot N = u^{-1}\cdot( \mu(z^{-1})\cdot N)=u^{-1}\cdot N\in \dot{w}\cdot H
		\]  
so $X_y \cap \B(N,H)$ is stable under this $\C^*$-action.  Now $X_y\cap \B(N,H)$ is smooth since $X_y\cap \B(N,H)\cong X_w^{\lambda}\cap \B(N,H)$, the $\C^*$-fixed points of smooth variety $X_w\cap \B(N,H)$.  Thus we can apply Lemma \ref{vb1} to $\pi_{\mu}$ to get a trivial vector sub-bundle $\pi_{\mu}: X_y \cap \B(N,H) \to \{\dot{y}\cdot \fb_L\}$.  Using the identification $X_y\cong X_w^{\lambda}$ there is a tower of vector bundles
		\[
				\xymatrix{ X_w\cap \B(N,H) \ar[d]^{\pi_{\lambda}} \\ 																			X_w^{\lambda} \cap \B(N,H) \ar[d]^{\pi_{\mu}}\\ \{\dot{w}\cdot \fb\} }
		\]
over the fixed point $\dot{w}\cdot \fb$.  The composition must be trivial, so $X_w\cap \B(N,H) \cong \C^{d}$ for some $d\in \Z$.  The result now follows from Remark \ref{paveit}.
\end{proof}

\begin{rem}  If $G=SL_n(\C)$ then Theorem \ref{paving} proves that $\B(N,H)$ is paved by affines for each Hessenberg space $H$ and nilpotent element $N\in \mathfrak{sl}_n(\C)$.  Indeed, given $N$ let $d_1\geq \cdots \geq d_k$ be the size of its Jordan blocks.  Then when $N$ is in Jordan form, it is regular in the standard Levi subalgebra $\mathfrak{sl}_{d_1}(\C) \times \cdots \times \mathfrak{sl}_{d_k}(\C)$ of $\fg$.  
\end{rem}

\begin{rem}  There are nilpotent elements in simple Lie algebras $\fg$, not of type $A$, which are not regular in a Levi subalgebra, such as any {\it distinguished} element of $\fg$ which is not regular.  When $\fg$ is the complex symplectic algebra of dimension $2n$ then a nilpotent element is distinguished if the sizes of its Jordan blocks consist of distinct even parts and regular if its Jordan form consists of a single block of dimension $2n$.  In general, if a nilpotent element is distinguished but not regular in $\fg$, or in a Levi subalgebra of $\fg$, it will not satisfy the assumptions of Theorem \ref{paving}.  Therefore while this theorem generalizes the results of Tymoczko in \cite{T2} to a larger collection of nilpotent elements, there are still interesting cases to be considered.  
\end{rem}

To illustrate our method, we compute the dimension of the affine cells paving the Hessenberg variety associated to a regular nilpotent element.

\begin{cor} \label{Nreg}  Let $N$ be a regular nilpotent element of $\fg$.  Fix a Hessenberg space $H$ with respect to $\fb$.  Then for all $w\in W$, $X_w\cap \B(N,H)$ is nonempty if and only if $\Delta\subset w(\Phi_H)$.  When $X_w\cap \B(N,H)$ is nonempty,
		\[
				\dim \left( X_w\cap \B(N,H) \right)= |\Phi_w\cap w(\Phi_H^-)|,
		\]
where $\Phi_H^-=\Phi^-\cap \Phi_H$.
\end{cor}
\begin{proof}  First $N \in \dot{w}\cdot H$ if and only if $X_w\cap \B(N,H)$ is nonempty.  But $N$ is the sum of all simple root vectors with respect to the fixed Borel subalgebra $\fb$, so $N\in \dot{w}\cdot H$ if and only if $\Delta \subset w(\Phi_H)$.  To calculate the dimension of the nonempty set $X_w\cap \B(N,H)$, recall that in this case $\V=\bigoplus_{\gamma\in \Phi^+ - \Delta} \fg_{\gamma}$ by Remark \ref{Nreg2}.  Thus by Proposition \ref{smooth},
		\begin{eqnarray*}
				\dim X_w\cap \B(N,H) &=& |\Phi_w| - \dim \V/(\V \cap \dot{w}\cdot H)\\
								&=& |\Phi_w| - |\{ \gamma\in \Phi^+ -\Delta : w^{-1}(\gamma)\notin \Phi_H \}|\\
								&=&|\Phi_w| - |\{ \gamma\in \Phi_w : w^{-1}(\gamma)\notin \Phi_H^- \}|\\
								&=& |\{ \gamma\in \Phi_w : w^{-1}(\gamma)\in \Phi_H^- \}|\\
								&=& |\Phi_w \cap w(\Phi_H^-)|
		\end{eqnarray*}
as required.
\end{proof}


\section{The arbitrary case, $X\in \fg$} \label{arbitrary}


We extend the affine paving result of the previous section to many Hessenberg varieties $\B(X,H)$ where $X\in \fg$ is not necessarily nilpotent.  Let $X=S+N$ be the Jordan decomposition of $X$ and $M=Z_G(S)$.  Then $M$ is a Levi subgroup of $G$ whose Lie algebra $\fm$ contains $X$.  

Suppose $N$ is regular in some Levi subalgebra of $\fm$.  By Corollary \ref{Uorb3} there exists a standard Borel subalgebra $\fb_M$ of $\fm$ so that
		\[
				U_M \cdot N=N+\V_N
		\]
where $N\in \fu_M$ and $\V_N \subset \fu_M$ is a direct sum of root spaces so that $N \notin \V_N$.  Since $S$ is in the center of $\fm$, $S\in \ft$ where $\ft$ is a fixed standard Cartan subalgebra of $\fm$.

Fix a Borel subalgebra $\fb$ of $\fg$ so that $\fm$ is standard and $\fb\cap \fm=\fb_M$ is the standard Borel subalgebra of $\fm$ above.  Then $X=S+N$ is in standard position with respect to $(\fb,\ft)$.  Let $Q=MU_Q$ denote the standard parabolic associated to $M$ in this basis.
 
\begin{lem} \label{Uorb2}  $U\cdot X=X+\V$ where $\V= \V_N \oplus \fu_Q$. 
\end{lem}
\begin{proof}  The $U$-orbit of $S$, the semisimple part of $X$, is $\fu_Q$.  Indeed, $\fm=\fg^S$ and $\fu_Q=\bigoplus_{\gamma\in \Phi^+, \, \gamma(S)\neq 0} \fg_{\gamma}$.  For all $Y\in \fu$, $ad_Y: \fu_Q \to \fu_Q$ and $[Y,S]\in \fu_Q$, so $U\cdot S \subset S+\fu_Q$.  In addition,
		\begin{eqnarray*}
				\dim U- \dim Z_U(S) &=& \dim \fu -\dim \fu^S \\
							         &=& \dim \fu- \dim (\fm\cap \fu)\\
							         &=& \dim \fu - \dim \fu_M\\
							         &=& \dim \fu_Q
		\end{eqnarray*}
implying $U\cdot S=S+\fu_Q$ by Remark \ref{unipotent}.

Certainly $U\cdot X\subset X+\V$.  Consider $\fu = \fu_M \oplus \fu_Q$.  Since $X\in \fm$, $\fu^{X}=\fu_M^{X} \oplus \fu_Q^{X}$.  By properties of the Jordan form, $\fg^X=\fg^S\cap \fg^N$.  Thus $\fu_Q^{X}=\fu_Q^{S}\cap \fu_Q^{N}=\{0\}$ since $\fu_Q^{S}=\{0\}$.  Similarly, $\fu_M^{X}=\fu_M^{S}\cap \fu_M^{N}=\fu_M \cap\fu_M^{N}= \fu_M^{N}$.  Now,
		\begin{eqnarray*}
				\dim U - \dim Z_{U}(X) &=& \dim \fu - \dim \fu^{X} \\
								&=& \dim \fu_M - \dim \fu_M^{X} +\dim \fu_Q - \dim \fu_Q^{X}\\
								&=& \dim \fu_M - \dim \fu_M^{N} + \dim \fu_Q\\
								&=& \dim \V_N + \dim \fu_Q
		\end{eqnarray*}
so $U\cdot X=X+\V$.
\end{proof}

\begin{prop}  Let $H$ be a fixed Hessenberg space in $\fg$ with respect to $\fb$.  Then for each $v\in W^M$, $H_v:=\dot{v}\cdot H \cap \fm$ is a Hessenberg space in $\fm$ with respect to $\fb_M$.
\end{prop}
\begin{proof}  We have only to show that $\Phi_M^+\subseteq \Phi_{H_v}$ and that $\Phi_{H_v}$ is closed under addition with roots from $\Phi_M^+$.  Note that $\Phi_{H_v}=v(\Phi_H)\cap \Phi_M$.  Let $\alpha\in\Phi_M^+$ and write $v^{-1}(\alpha)=\gamma$, so $\alpha=v(\gamma)$ for some $\gamma\in\Phi^+$ since $\Phi_v \cap \Phi^+_M=\emptyset$.

First, $\gamma\in\Phi^+\subseteq \Phi_H$ and therefore $\alpha=v(\gamma)\in v(\Phi_H)$ and $\alpha\in \Phi_M$ implying that $\alpha\in \Phi_{H_v}$.  Thus $\Phi_M^+\subseteq \Phi_{H_v}$.  Next, let $v(\beta)\in \Phi_{H_v}$ such that $\alpha+v(\beta)$ is a root of $\Phi_M$.  Then 
		\[
				\alpha+v(\beta)=v(\gamma)+v(\beta)=v(\gamma+\beta)\in v(\Phi_H)
		\]
since $\gamma\in \Phi^+, \beta\in \Phi_H$ and $\Phi_H$ is closed under addition of roots from $\Phi^+$.  So $\alpha+v(\beta)\in v(\Phi_H)\cap \Phi_M= \Phi_{H_{v}}$, i.e. $\Phi_{H_v}$ is closed with respect to addition with roots from $\Phi_M^+$.
\end{proof}

Let $Z=Z_G(M)^0$ as in Section \ref{fixedpts} and let $\mu: \C^* \to T$ be a dominant 1-parameter subgroup such that $\B^Z=\B^{\mu}$.  The $Z$-action on $\B$ restricts nicely to the Hessenberg variety in the following sense.

\begin{prop} \label{reduction}  Fix $X\in \fg$ with Jordan decomposition $X=S+N$.  Let $H$ be a Hessenberg space of $\fg$ with respect to $\fb$ and let $w\in W$ have decomposition $w=yv$ where $y\in W_M$ and $v\in W^M$.  Then the isomorphism $X_w^{\mu}=X_w^Z\cong X_y$ given in Remark \ref{isom} restricts to an isomorphism
		\begin{eqnarray*}
				X_w^Z \cap \B(X,H) &\to& X_y\cap\B(N,H_v)\\
					u\dot{w} \cdot \fb &\mapsto& u\dot{y} \cdot \fb_M
		\end{eqnarray*}
where $\B(N,H_v)$ is the Hessenberg variety in $\B(M)$ associated to nilpotent element $N\in \fm$ and Hessenberg space $H_v$.
\end{prop}
\begin{proof}  We must show that $u \dot{w} \cdot \fb \in \B(X,H)$ if and only if $u \dot{y} \cdot \fb_M \in \B(N,H_v)$ for all $u\in U^y$.  We have
		\begin{eqnarray*}
				u^{-1}\cdot X \in \dot{w}\cdot H &\Leftrightarrow& \dot{y}^{-1}u^{-1}\cdot X\in \dot{v}\cdot H\\
									&\Leftrightarrow& \dot{y}^{-1}u^{-1} \cdot S + \dot{y}^{-1}u^{-1}\cdot N \in \dot{v}\cdot H\\
									&\Leftrightarrow&  S+ \dot{y}^{-1}u^{-1}\cdot N \in \dot{v}\cdot H\\
									&\Leftrightarrow& \dot{y}^{-1}u^{-1}\cdot N \in \dot{v}\cdot H\cap \fm\\
									&\Leftrightarrow& u^{-1}\cdot N \in \dot{y}\cdot H_v			
		\end{eqnarray*}
since $S\in \dot{v}\cdot H$ for all $v\in W^M$ and $\dot{y},u \in M, N\in \fm$ implies $\dot{y}^{-1}u^{-1}\cdot N\in \fm$.
\end{proof}

\begin{thm} \label{pavingX}  Suppose $X\in \fg$ has Jordan decomposition $X=S+N$ and $N$ is regular in some Levi subalgebra of $\fm$, where $\fm$ is the Lie algebra of Levi subgroup $M=Z_G(S)$.  Then $\B(X,H)$ is paved by affines.
\end{thm}
\begin{proof}  Fix a Hessenberg space $H$ with respect to $\fb$.  By Lemma \ref{Uorb2} there exists a direct sum of root spaces $\V\subset \fu$ such that $X\notin \V$ and $U\cdot X =X+\V$.  Therefore by Proposition \ref{smooth}, $X_w\cap \B(X,H)\neq \emptyset$ if and only if $N\in \dot{w}\cdot H$ and when $X_w\cap \B(X,H)\neq \emptyset$, the intersection is smooth.  Recall that equation (\ref{vb}) exhibits a vector bundle $\pi_{\mu}: X_w \to X_w^{\mu}$ with fiber preserving a strictly positive $\C^*$-action induced by $\mu$.  Since $\fm=\fg_0(\mu)$ and $X\in \fm$ this $\C^*$-action fixes $X$ and therefore the intersection $X_w\cap \B(X,H)$ is $\C^*$-stable.  Apply Lemma \ref{vb1} to get a vector sub-bundle $\pi_{\mu}: X_w\cap \B(X,H) \to X_w^{\mu}\cap \B(X,H)$.  

By Proposition \ref{reduction}, $X_w^{\mu}\cap \B(X,H) \cong X_y\cap \B(N, H_v)$ where $\B(N,H_v)$ is the Hessenberg variety associated to the nilpotent element $N\in \fm$ and Hessenberg space $H_v$ with resepect to $\fb_M$.  By assumption, $N$ is regular in some Levi subalgebra of $\fm$.  Therefore by the proof of Theorem \ref{paving}, $X_y\cap \B(N, H_v)$ is the total space of a trivial vector bundle over $\{ \dot{y} \cdot \fb_M \}$.  Using the identification $X_w^{\mu}\cong X_y$ given in Remark \ref{isom}, there is a tower of vector bundles
		\[
				\xymatrix{ X_w\cap \B(X,H) \ar[d]^{\pi_{\mu}} \\ X_w^{\mu}\cap \B(X,H) \ar[d]\\ \{ \dot{w}\cdot \fb \} }
		\]
over the fixed point $\dot{w}\cdot \fb$.  The composition must be trivial, so $X_w\cap\B(X,H) \cong \C^{d}$ for some $d\in \Z$.  Now the result follows from Remark \ref{paveit}.  
\end{proof}

\begin{cor} \label{dimen}  Suppose $X\in \fg$ has Jordan decomposition $X=S+N$ and satisfies the conditions of Theorem \ref{pavingX}.  Fix a Hessenberg space $H$ with respect to $\fb$.  Then if $X_w\cap \B(N,H)$ is nonempty, it has dimension
		\[
				 \dim \left( X_y\cap\B(N,H_v)\right)+|y(\Phi_v) \cap w (\Phi_H^-)|
		\]
where $w=yv$ for $y\in W_M$ and $v\in W^M$
\end{cor}
\begin{proof}  To compute the dimension of the nonempty set $X_w\cap \B(N,H)$ recall that $\V=\V_N \oplus \fu_Q$.  Now $\V\cap \dot{w}\cdot H=(\V_N \cap \dot{y}\cdot H_v )\oplus (\fu_Q\cap \dot{w}\cdot H) $ so by Proposition \ref{smooth} and Corollary \ref{roots}
		\begin{eqnarray*}
				\dim X_w\cap \B(X,H) &=& |\Phi_w| - \dim \V / (\V\cap \dot{w}\cdot H)\\
								  &=& |\Phi_y| - \dim \V_N/(\V_N \cap \dot{y}\cdot H_v) + |y(\Phi_v)| - 															\dim \fu_Q/(\fu_Q\cap \dot{w}\cdot H). 
		\end{eqnarray*}
A second application of Proposition \ref{smooth} yields the equality 
		\[
				|\Phi_y| - \dim \V_N/(\V_N \cap \dot{y}\cdot H_v)=\dim \left(X_y \cap \B(N,H_v)\right).
		\]
Finally, $|y(\Phi_v)| - \dim \fu_Q/(\fu_Q\cap \dot{w}\cdot H)= |y(\Phi_v) \cap w(\Phi_H^-)|$ by a calculation similar to that in the proof of Corollary \ref{Nreg}.
\end{proof}

\begin{rem}  Theorem \ref{pavingX} applies to Hessenberg varieties $\B(X,H)$ when $X$ is a semisimple element and when $X$ is a regular element.  Indeed, if $X$ is semisimple, then $N=0$ is a regular element of $\ft\subset \fm$.  If $X$ is regular, then $N$ is a regular element of $\fm$.  In both cases $X\in \fg$ satisfies the assumptions of the Theorem.
\end{rem}

\begin{rem}  Theorem \ref{pavingX} gives an affine paving of the Springer variety $\B^X=\B(X,\fb)$ when $X\in \fg$ satisfies the assumptions of the Theorem.
\end{rem}

\begin{cor}  Suppose $X\in \fg$ has Jordan decomposition $X=S+N$ and satisfies the assumptions of Theorem \ref{pavingX}.  Fix a Hessenberg space $H$ with respect to $\fb$.  Then for all $w=yv$ where $y\in W_M$ and $v\in W^M$ we have the following.
		\begin{enumerate}
				\item If $N=0$, i.e. if $X$ is a semisimple element, then $X_w\cap \B(X,H)$ is nonempty for all $w\in W$ and 
							\[
									\dim \left( X_w\cap \B(X,H) \right) = |\Phi_y| + |y(\Phi_v) \cap w(\Phi_H^-)|.
							\]
				\item If $N$ is regular in $\fm$, i.e. if $X$ is a regular element, then $X_w\cap \B(X,H)$ is nonempty if and only if 					$\Delta_M \subset y(\Phi_{H_v})$.  When $X_w\cap \B(N,H)\neq \emptyset$,
							\[
									\dim \left( X_w\cap \B(X,H) \right) 																					= |\Phi_y \cap y(\Phi^-_{H_v})|+|y(\Phi_v) \cap w(\Phi_H^-)|.
							\]
					
		\end{enumerate}
\end{cor} 
\begin{proof}  First, by Proposition \ref{smooth}, if $N\in \dot{w}\cdot H$ then $X_w\cap \B(X,H)$ is nonempty.  When $X$ is semisimple, $N=0\in \dot{w}\cdot H$ for all $w\in W$ and $\B(N,H_v)=\B(M)$, so (1) is a direct consequence of Corollary \ref{dimen}.  For part (2), $N\in \dot{w}\cdot H$ if and only if $N\in \dot{y}\cdot H_v$ using the identification given in Proposition \ref{reduction}.  Therefore $X_w\cap \B(X,H)$ is nonempty if and only if $N\in \dot{y}\cdot H_v$.  The statement now follows from Corollary \ref{Nreg} and Corollary \ref{dimen}. 
\end{proof}

%

%


\begin{thebibliography}{100000}

\bibitem[1]{BGG} 
I. N. Bernstein, I.M. Gel'fand, and S.I. Gel'fand, 
\textit{Schubert cells and cohomology of the spaces $G/P$}, 
Uspehi Mat. Nauk {\bf 28} (1973),
no. 3(171), 3-26.

\bibitem[2]{BH}
H. Bass and W. Haboush,
{\em Linearizing certain reductive group actions}, 
Trans. Amer. Math. Soc. {\bf 292} (1985), 463-482.

\bibitem[3]{CG}
N. Chriss and V. Ginzburg, 
{\em Representation Theory and Complex Geometry},
Birkh\"{a}user, Boston, 1997.

\bibitem[4]{dCLP}
\newblock 
C. de Concini, G. Lusztig, and C. Procesi, 
\newblock \emph{Homology of the zero-set of a nilpotent vector field on a flag manifold}, 
\newblock J. Amer. Math. Soc. \textbf{1} (1988), 15-34.

\bibitem[5]{dMPS}
F. de Mari, C. Procesi, and M. A. Shayman, {\em Hessenberg Varieties}, Trans. Amer. Math. Soc. {\bf 332} (1992),
529-534.

\bibitem[6]{Ha}
R. Hartshorne,
{\em Algebraic Geometry},
Springer-Verlag, New York, 1977.

\bibitem[7]{H}
J. Humphreys, {\em Linear Algebraic Groups}, Springer-Verlag, New York, 1964.

\bibitem[8]{K} 
B. Kostant,
{\em Lie algebra cohomology and the generalized Borel-Weil Theorem},
Ann. of Math. (2) {\bf 74} (1961), 329-387.

\bibitem[9]{T}
J. Tymoczko, {\em Linear conditions imposed on flag varieties},
Amer. J. Math. {\bf 128} (2006), 1587-1604

\bibitem[10]{T2}
J. Tymoczko, {\em Paving Hessenberg Varieties by affines},
Sel. math., New ser. {\bf 13} (2007), 353-367.

\bibitem[11]{T3}
J. Tymoczko,
{\em Decomposing Hessenberg varieties over classical groups},
Dissertation, Princeton University,  arXiv:math/0211226.

\end{thebibliography}
\end{document}